\documentclass[12pt]{amsart}
\usepackage{amssymb}
\usepackage{amsmath, amscd}
\usepackage{amsthm}
\usepackage{amsmath}
\usepackage{comment}
\usepackage[table,xcdraw]{xcolor}
\usepackage{ulem}
\usepackage{tikz}
\usepackage{float}
\usepackage{graphicx}
\usepackage{circuitikz}
\usetikzlibrary{positioning}
\usepackage[colorlinks=true, linkcolor=blue,urlcolor=blue]{hyperref}
\usepackage{nicematrix}

\newtheorem{theorem}{Theorem}[section]

\newtheorem{proposition}[theorem]{Proposition}
\newtheorem{lemma}[theorem]{Lemma}

\newtheorem{corollary}[theorem]{Corollary}
\theoremstyle{definition}

\newtheorem{example}[theorem]{Example}
\newtheorem{definition}[theorem]{Definition}

\newtheorem{problem}[theorem]{Problem}

\date{\today}

\begin{document}
	
\author[P. Danchev]{P. Danchev}
\address{Institute of Mathematics and Informatics, Bulgarian Academy of Sciences, 1113 Sofia, Bulgaria}
\email{danchev@math.bas.bg; pvdanchev@yahoo.com}

\author[O. Hasanzadeh]{O. Hasanzadeh}
\address{Department of Mathematics, Tarbiat Modares University, 14115-111 Tehran Jalal AleAhmad Nasr, Iran}
\email{o.hasanzade@modares.ac.ir; hasanzadeomiid@gmail.com}	
	
\author[A. Moussavi]{A. Moussavi}
\address{Department of Mathematics, Tarbiat Modares University, 14115-111 Tehran Jalal AleAhmad Nasr, Iran}
\email{moussavi.a@modares.ac.ir; moussavi.a@gmail.com}

\author[A. Javan]{A. Javan}
\address{Department of Mathematics, Tarbiat Modares University, 14115-111 Tehran Jalal AleAhmad Nasr, Iran}
\email{a.darajavan@modares.ac.ir; a.darajavan@gmail.com}

\title[$\sqrt{\Delta}$-fine rings]{$\sqrt{\Delta}$-Fine Rings}

\maketitle

\begin{abstract}
We introduce and study the so-termed {\it $\sqrt{\Delta}$-fine rings}, a new class of rings that generalizes the classical {\it fine rings} introduced by C\u{a}lug\u{a}reanu-Lam in J. Algebra \& Appl. (2016) by requiring that every nonzero element $r \in R$ can be written as $r = u + a$, where $u$ is a unit and $a \in \sqrt{\Delta(R)}$. We establish that every such ring is simple, every abelian $\sqrt{\Delta}$-fine ring is indecomposable, and most notably, the matrix ring $M_n(R)$ over a $\sqrt{\Delta}$-fine ring $R$ is again $\sqrt{\Delta}$-fine for every $n \ge 1$. As a consequence, we characterize all semi-local $\sqrt{\Delta}$-fine rings as those rings which are precisely the simple Artinian rings. We also examine group rings, providing conditions under which they are either $\sqrt{\Delta}$-fine or generalized fine, where the latter class was introduced by Zhou in J. Algebra \& Appl. (2022), and conclude our work with the difficult open question asking of whether each $\sqrt{\Delta}$-fine ring is necessarily fine.
\end{abstract}

\section{Introduction and Motivation}

Throughout this paper, \( R \) is an arbitrary associative ring with identity which need {\it not} be commutative. We, hereafter, use the usual notations: \( U(R) \) for the group of units, \( \mathrm{Nil}(R) \) for the set of nilpotent elements, \( C(R) \) for the center, \( \mathrm{Id}(R) \) for the set of idempotent elements, and \( J(R) \) for the Jacobson radical. For each positive integer \( n \), we write \( M_n(R) \) for the ring of all \( n \times n \) matrices over \( R \).

As a point of departure of our considerations, note that Nicholson \cite{nic} defined the concept of clean rings which arose from the study of rings with the exchange property: A ring is called {[\it clean} if every element can be written as a sum of a unit and an idempotent. Over the past two decades, this class of rings has attracted considerable attention, along with various generalizations and related notions.

Moreover, Diesl \cite{di} defined the concept of (strongly) nil-clean rings: A ring \( R \) is called {\it nil-clean} if every element \( r \in R \) can be expressed as \( r = b + e \), where \( b \in \mathrm{Nil}(R) \) and \( e \in \mathrm{Id}(R) \). In addition, in the {\it strongly nil-clean} case, the extra condition \( eb = be \) is requested. Several fundamental properties and a comprehensive theory of (strongly) nil-clean rings were established there (compare also with \cite{DL}).

In accordance with C\u{a}lug\u{a}reanu and Lam \cite{Clam}, a nonzero element \( a \) in a ring \( R \) is said to be {\it fine} if it can be written as a unit plus a nilpotent element. A ring \( R \) is then also said to be {\it fine} if each nonzero element is fine. Among the other interesting results there, one of the main outcomes of \cite{Clam} is that any matrix ring over a fine ring is again fine.

Additionally, a new class of rings, named generalized fine rings, was recently introduced by Zhou in \cite{z}: A ring \( R \) is called {\it generalized fine} whenever each element outside its Jacobson radical can be expressed as a sum of a unit and a nilpotent. Again the property of being generalized fine is preserved under taking matrix rings, that is, if \( R \) is a generalized fine ring, then so is \( M_n(R) \) for every positive integer \( n \).

In our further exploration, we shall also make use of the set \(\Delta(R)\), which was first studied by Lam in \cite[Exercise 4.24]{lamm}, and later by Leroy and Matczuk in \cite{lm}. According to \cite[Theorems 3 and 6]{lm}, \(\Delta(R)\) is the largest subring of the Jacobson radical that is closed under multiplication by units (and also by quasi-invertible elements). We thus always have fulfilled \(J(R) \subseteq \Delta(R)\). Moreover, \(\Delta(R)\) equals the Jacobson radical of the subring generated by the units of \(R\), and it coincides with \(J(R)\) precisely when it is a two-sided ideal. Equivalently, an element \(a \in R\) belongs to \(\Delta(R)\) exactly when \(1 - ua\) is a unit for every unit \(u \in U(R)\).

In order of a possible non-trivial generalization, the main tool we use in this paper is the set \( \sqrt{\Delta(R)} \), which was introduced in \cite{DDE} and consists of all elements \( x \in R \) such that \( x^n \in \Delta(R) \) for some \( n \ge 1 \). This set properly contains \( \Delta(R) \), though it need not be a subring of \( R \). In any case, the inclusion \( \mathrm{Nil}(R) \subseteq \sqrt{\Delta(R)} \) always holds.

In the other vein, V\'{a}mos \cite{v2g}defined the so-called {\it \( 2 \)-good rings} that are rings in which each element is a sum of two units.

In the present article, we define a new class of rings by replacing the nilpotents in fine rings with elements from the more general set \( \sqrt{\Delta(R)} \). So, we say that a ring \( R \) is {\it \( \sqrt{\Delta} \)-fine} if every nonzero element can be written as a sum of a unit and an element from \( \sqrt{\Delta(R)} \). This definition appears as a natural expansion of the notion of fine rings.

Our plan to investigate these rings is organized as follows: In Section \ref{sec2}, some fundamental properties of \( \sqrt{\Delta} \)-fine rings are obtained. We shall prove in what follows that each \( \sqrt{\Delta} \)-fine ring is simple (see Proposition \ref{simple}) as well as that any abelian $\sqrt{\Delta}$-fine ring is indecomposable (see Proposition \ref{centralidempotent}). Likewise, for all $\sqrt{\Delta}$-fine rings, we always have \(R = \{1\} \cup \bigl(U(R) + U(R)\bigr)\), and \(R\) is $2$-good if, and only if, \(|R| \neq 2\) (see Theorem \ref{2-good}).

Furthermore, in Section \ref{sec3}, a key result says that matrix rings over $\sqrt{\Delta}$-fine rings are too $\sqrt{\Delta}$-fine (see Theorem \ref{main theorem}). As a consequence, we show that a ring \( R \) is simple Artinian if, and only if, it is simultaneously semi-local and \( \sqrt{\Delta} \)-fine (see Corollary \ref{artinian}). We also demonstrate that a semi-simple ring is \( \sqrt{\Delta} \)-fine uniquely when it is simple (see Lemma \ref{semisimple}).

Finally, in Section \ref{sec4}, we focus on group rings and determine certain conditions under which they are \( \sqrt{\Delta} \)-fine and generalized fine (see Lemma \ref{8}, Theorem \ref{88} and Theorem \ref{888}, respectively).

However, despite our efforts, we were unable to find an example of a \(\sqrt{\Delta}\)-fine ring that is {\it not} fine. This logically leads us to pose the following unsettled problem: Does there exist a ring that is \(\sqrt{\Delta}\)-fine but {\it not} fine (see Problem~\ref{basic})? Nevertheless, we succeeded to construct an explicit example of a generalized fine ring which is {\it not} \(\sqrt{\Delta}\)-fine (see Example~\ref{gfi}).

\section{Basic Properties of $\sqrt{\Delta}$-Fine Rings}\label{sec2}

The aim of this section is to illustrate some elementary properties of $\sqrt{\Delta}$-fine rings.

\begin{definition}\label{def}
Let \(R\) be a ring. A nonzero element \(r\in R\) is called
\begin{itemize}
\item $\sqrt{\Delta}$-fine if \(r = u+a\) for some \(a\in\sqrt{\Delta(R)}\) and \(u\in U(R)\); Henceforth, such a decomposition is said to be a $\sqrt{\Delta}$-fine decomposition.
\item strongly $\sqrt{\Delta}$-fine if, in addition, \(au = ua\);
\item uniquely $\sqrt{\Delta}$-fine if $r$ has a unique $\sqrt{\Delta}$-fine decomposition.
\end{itemize}
The ring \(R\) is said to be $\sqrt{\Delta}$-fine (respectively, strongly $\sqrt{\Delta}$-fine, uniquely $\sqrt{\Delta}$-fine) if every nonzero element of \(R\) has the corresponding property.
\end{definition}

We begin our studies with the following.

\begin{proposition}\label{1}
For any ring \(R\), the following statements are true:
	
(1) An element \(r\in R\setminus\{0\}\) is strongly $\sqrt{\Delta}$-fine if, and only if, \(r\in U(R)\).
	
(2) \(R\) is strongly $\sqrt{\Delta}$-fine if, and only if, \(R\) is a division ring.
	
(3) \(R\) is $\sqrt{\Delta}$-fine and every unit is central if, and only if, it is a field.
	
(4) \(R\) is uniquely $\sqrt{\Delta}$-fine if, and only if, \(R\) is a division ring.
\end{proposition}

\begin{proof}
(1) One knows that units in a nonzero ring are always strongly $\sqrt{\Delta}$-fine. Conversely, assume that a nonzero element \(r\) admits a strongly $\sqrt{\Delta}$-fine decomposition \(r=u+a\) with \(u\in U(R)\), \(a\in\sqrt{\Delta(R)}\) and \(ua=au\). Then, it must be that \(r=u(1+u^{-1}a)\). In conjunction with \cite[Lemma 2.1(1)]{DDE}, the commutativity of \(u\) and \(a\) forces that \(u^{-1}a\in \sqrt{\Delta(R)}\), whence \(1+u^{-1}a\in U(R)\) with the help of \cite[Lemma 2.1(2)]{DDE}. Consequently, \(r\in U(R)\).
	
(2) If \(R\) is a division ring, then every nonzero element is a unit, and therefore strongly $\sqrt{\Delta}$-fine by (1). Conversely, if \(R\) is strongly $\sqrt{\Delta}$-fine, then every nonzero element is a unit by (1) and hence \(R\) is a division ring.
	
(3) The ``if'' part is clear. For the converse one, suppose \(R\) is $\sqrt{\Delta}$-fine and each unit is central. Because units are central, each $\sqrt{\Delta}$-fine decomposition is automatically strongly $\sqrt{\Delta}$-fine by Definition \ref{def}; thus, \(R\) is strongly $\sqrt{\Delta}$-fine. But, in view of (2), it is a division ring, and since all units are central it is necessarily commutative, i.e., a field.
	
(4) Again the ``if'' part is obvious. For the reciprocal one, assume that \(R\) is uniquely $\sqrt{\Delta}$-fine. It suffices to show that \(\sqrt{\Delta(R)}=0\). In fact, choose \(a\in\sqrt{\Delta(R)}\) and note that \(1+a\neq0\) as \(R\neq0\). From the two distinct $\sqrt{\Delta}$-fine decompositions
$
1+a = (1+a)+0,
$
the uniqueness forces \(a=0\). Hence, \(\sqrt{\Delta(R)}=0\); in particular, \(R\) is a division ring, as required.
\end{proof}

We now define
\[
\sqrt{\Delta}F(R)=\left\{ r \in R \setminus \{0\} \mid \exists u \in U(R),\; a \in \sqrt{\Delta(R)},\text{ such that } r=u+a \right\}
\]

\noindent and proceed by proving the next technicalities.

\begin{proposition}\label{2}
For any ring \(R\), the following claims hold:
	
(1) If \(R \neq 0\), then \(\sqrt{\Delta}F(R) \cap C(R) = U(C(R))\).
	
(2) \(v \in U(R)\) yields \(v\sqrt{\Delta}F(R)v^{-1} = \sqrt{\Delta}F(R)\).
If, moreover, \(v \in U(R) \cap C(R)\), then \(v\sqrt{\Delta}F(R) = \sqrt{\Delta}F(R)\).
	
(3) If \(R \neq 0\) and \(\sqrt{\Delta(R)} \subseteq \Delta(R)\) (i.e., \(\sqrt{\Delta(R)} = \Delta(R)\)), then \(\sqrt{\Delta}F(R) = U(R)\).
\end{proposition}

\begin{proof}
(1) We know that \(U(C(R))\subseteq \sqrt{\Delta}F(R) \cap C(R)\), giving one containment. 

Conversely, take \(r \in  \sqrt{\Delta}F(R)\cap C(R)\) and consider a $\sqrt{\Delta}$-fine decomposition \(r = u + a\). Since \(u\) commutes with \(a\), we have \(u^{-1}a\in \sqrt{\Delta(R)}\) via \cite[Lemma 2.1(1)]{DDE}, and hence $1 + u^{-1}a\in U(R)$ via \cite[Lemma 2.1(2)]{DDE}. So,
\[
	r = u(1 + u^{-1}a) \in U(R) \cdot U(R) \subseteq U(R).
\]
Thus, \(r \in U(C(R))\).
	
(2) Let $d\in \sqrt{\Delta(R)}$. It is easy to see that $vdv^{-1}\in \sqrt{\Delta(R)}$ by \cite[Lemma 2.1(10)]{DDE}. So, the first statement is obvious. For the second one, if \(r = u + a\) is a $\sqrt{\Delta}$-fine decomposition of \(r \in \sqrt{\Delta}F(R)\) and \(v \in U(R) \cap C(R)\), then
\[
	vr = vu + va
\]
is a $\sqrt{\Delta}$-fine decomposition of \(va\); whence \(v\sqrt{\Delta}F(R) \subseteq \sqrt{\Delta}F(R)\). Replacing \(v\) by \(v^{-1}\) gives the desired equality.
	
(3) Under given hypothesis, take any \(r \in R\) with $\sqrt{\Delta}$-fine decomposition \(r = u + a\). Then, \(a \in \Delta(R)=\sqrt{\Delta(R)}\). Consequently,
\[
	r \in U(R) + \Delta(R) \subseteq U(R),
\]
by \cite[Lemma 1]{Dj} and, therefore, \(\sqrt{\Delta}F(R) = U(R)\), as claimed.
\end{proof}

\begin{proposition}\label{quotient}
Let \( I \) be an ideal of a \(\sqrt{\Delta}\)-fine ring \( R \). Then, \( R/I \) is also \(\sqrt{\Delta}\)-fine.
\end{proposition}

\begin{proof}
Take any nonzero element \( r + I \in R/I \), i.e., \( r \notin I \). In particular, \( r \neq 0 \) (for if \( r = 0 \), then \( r + I = I \), contradicting the assumption that \( r + I \) is nonzero). Since \( R \) is \(\sqrt{\Delta}\)-fine, there exist \( u \in U(R) \) and \( a \in \sqrt{\Delta(R)} \) such that \( r = u + a \). Projecting onto the quotient gives
$
	r + I = (u + I) + (a + I).
$
Evidently, \( u + I \) is a unit in \( R/I \), and \( a + I \in \sqrt{\Delta(R/I)} \). Hence, \( r + I \) admits a \(\sqrt{\Delta}\)-fine decomposition, showing that \( R/I \) is \(\sqrt{\Delta}\)-fine, as wanted.
\end{proof}

\begin{lemma}\label{frac}
Let \(I\subseteq J(R)\) be an ideal of a ring \(R\), and let \(r \in R\). Then, \(r \in \sqrt{\Delta}F(R)\) if, and only if, \(r + I \in \sqrt{\Delta}F(R/I)\).
\end{lemma}

\begin{proof}
The necessity is straightforward. To prove sufficiency, assume \(r + I \in \sqrt{\Delta}F(R/I)\). Then, there exist \(u+I \in U(R/I)\) and \(a+I \in \sqrt{\Delta(R/I)}\) such that \(r + I = (u+I) + (a+I)\). Hence, $r-(u+a)\subseteq I\subseteq J(R)\subseteq \Delta(R)$. Therefore, $r=u+(a+d)$ for some $d\in \Delta(R)$. It is obvious that $u\in U(R)$. Also, $a+d\in \sqrt{\Delta(R)}$ in virtue of \cite[Lemma 2.1(7)]{DDE}, as needed.
\end{proof}

\begin{proposition}\label{simple}
If $r\in \sqrt{\Delta}F(R)$, then $RrR=R$. In particular, any $\sqrt{\Delta}$-fine ring $R$ is a simple ring.
\end{proposition}

\begin{proof}
Let \( r = u + a\) be a $\sqrt{\Delta}$-fine decomposition in a ring \(R\). In the factor ring \(\overline{R} = R / RrR\), the image \(\overline{u} = -\overline{a}\) is both a unit and in $\sqrt{\Delta(\overline{R})}$ by \cite[Lemma 2.1(4)]{DDE} and Proposition \ref{quotient}. Consequently, \(\overline{R} = \{0\}\) thanks to \cite[Remark 2.2(4)]{DDE}; that is, \(RrR = R\). Since \(r\) was arbitrary, every nonzero element of \(R\) generates \(R\) as a two-sided ideal. Thus, \(R\) has no nontrivial proper ideals; i.e., \(R\) is simple, as asserted.
\end{proof}

In \cite{wang}, the set
\[
\sqrt{J(R)} = \{ x \in R : x^{n} \in J(R) \text{ for some } n \ge 1 \}
\]
was defined.

\medskip

We are now prepared to establish the following assertions.

\begin{lemma}\label{3}
If \(f: R \rightarrow S\) is a surjective ring homomorphism, then $f(\sqrt{J(R)}) \subseteq \sqrt{J(S)}$.
\end{lemma}

\begin{proof}
Take \(t \in f(\sqrt{J(R)})\). Then, \(t = f(x)\) for some \(x \in \sqrt{J(R)}\), so there exists \(n \ge 1\) with \(x^n \in J(R)\). Consequently,
\[
	t^n = (f(x))^n = f(x^n) \in f(J(R)) \subseteq J(S),
\]
where the last inclusion follows from \cite[Exercise 4.10]{lam}. Hence, \(t \in \sqrt{J(S)}\), as asked.
\end{proof}

\begin{proposition}\label{centralidempotent}
No nontrivial central idempotent can be $\sqrt{\Delta}$-fine. In particular, abelian $\sqrt{\Delta}$-fine rings are indecomposable.
\end{proposition}

\begin{proof}
Let \(e\) be a nontrivial central idempotent in a \(\sqrt{\Delta}\)-fine ring \(R\), and suppose \(e = u + a\) with \(u \in U(R)\) and \(a \in \sqrt{\Delta(R)}\). Since \(e\) is central, we infer \(ua = au\). So, \(e = u(1 + u^{-1}a)\). Because \(u\) and \(a\) commute, we derive \(u^{-1}a \in \sqrt{\Delta(R)}\) with the aid of \cite[Lemma 2.1(1)]{DDE}, so that \(1 + u^{-1}a \in U(R)\) invoking to \cite[Lemma 2.1(2)]{DDE}. This, however, contradicts the assumption that \(e\) is nontrivial, thus substantiating our statement. The second part now follows immediately.
\end{proof}

Our first main result sounds thus.

\begin{theorem}\label{2-good}
For any $\sqrt{\Delta}$-fine ring \(R\), the following are valid:
	
(1) \(R = \{1\} \cup \bigl(U(R) + U(R)\bigr)\).
	
(2) \(R\) is 2-good if and only if \(|R| \neq 2\).
\end{theorem}

\begin{proof}
(1) Take any \(r \in R\). If \(r = 1\), we are done. Otherwise, \(r - 1 \in \sqrt{\Delta}F(R)\) and thus admits a fine decomposition \(r - 1 = u + a\) with \(u \in U(R)\) and \(a \in \sqrt{\Delta(R)}\). Hence,
\[
	r = u + (1 + a) \in U(R) + U(R),
\]
since \(1 + a \in U(R)\). This establishes (1).
	
(2) The ``if'' part is routine, so we omit its check. Conversely, suppose \(|R| > 2\) and assume, in way of contradiction, that \(R\) is 2-good. We first show that \(U(R) \neq \{1\}\). Indeed, if \(U(R) = \{1\}\), then each nonzero \(r \in R\) has a $\sqrt{\Delta}$-fine decomposition \(r = 1 + a\) with \(a \in \sqrt{\Delta(R)}\); consequently, \(r \in U(R)\) and hence \(r = 1\). This would force \(|R| = 2\), contradicting the hypothesis. Thus, there exists a unit \(v \neq 1\).
	
Now, consider a $\sqrt{\Delta}$-fine decomposition of \(1 - v\neq0\)
\[
	1 - v = w + b \quad (w \in U(R),\; b \in \sqrt{\Delta(R)}).
\]
Set \(x := w + v = 1 - b\). Since $b \in \sqrt{\Delta(R)}$, we arrive at \(x \in U(R)\). Therefore,
\[
	1 = x^{-1}w + x^{-1}v \in U(R) + U(R),
\]
which unambiguously shows that \(1\) is a sum of two units. Hence, \(R\) is 2-good.
\end{proof}

The following statement extends the corresponding one from \cite{Clam}

\begin{proposition}\label{UU}
A ring \(R\) is $\sqrt{\Delta}$-fine and \(U(R) = 1 +\sqrt{\Delta(R)}\) if, and only if, \(R \cong \mathbb{Z}_2\).
\end{proposition}

\begin{proof}
The ``if'' part is trivial. Reciprocally, assume \(R\) is both $\sqrt{\Delta}$-fine and \(U(R) = 1 +\sqrt{\Delta(R)}\). Suppose, for contradiction, that \(|R| > 2\). But, in the proof of Theorem~\ref{2-good}(2), we established that under these circumstances one can write \(1 = u_1 + u_2\) for some units \(u_1, u_2 \in U(R)\). Since \(R\) is \(U(R) = 1 +\sqrt{\Delta(R)}\), we find \(u_1 \in 1 + \sqrt{\Delta(R)}\) or, equivalently, \(u_1 - 1 \in \sqrt{\Delta(R)}\). However, one sees that \(u_1 - 1 = -u_2\) is also a unit. Thus, $u_1-1\in U(R)\cap \sqrt{\Delta(R)}$, a contradiction with \cite[Remark 2.2(4)]{DDE} at hand. Therefore, \(|R| = 2\) and hence \(R \cong \mathbb{Z}_2\), as promised.
\end{proof}

Two more examples show the following.

\begin{example}
The field \(\mathbb{Z}_2\) is $\sqrt{\Delta}$-fine. But, the direct product \(\mathbb{Z}_2 \times \mathbb{Z}_2\) is definitely {\it not} $\sqrt{\Delta}$-fine.
\end{example}

\begin{example}
A subring of a \(\sqrt{\Delta}\)-fine ring need {\it not} be \(\sqrt{\Delta}\)-fine. For instance, the matrix ring \(M_2(\mathbb{Z}_2)\) is $\sqrt{\Delta}$-fine owing to Theorem \ref{main theorem} quoted below. But, the triangular matrix ring \(T_2(\mathbb{Z}_2)\) is {\it not} $\sqrt{\Delta}$-fine despite being a subring of \(M_2(\mathbb{Z}_2)\), because \(\mathbb{Z}_2 \times \mathbb{Z}_2\) is viewed as a quotient of \(T_2(\mathbb{Z}_2)\).
\end{example}

We continue with the following statements.

\begin{proposition}\label{n}
If \(R\) is a \(\sqrt{\Delta}\)-fine ring with \(U(R) = \{1\}\), then \(R \cong \mathbb{Z}_2\).
\end{proposition}

\begin{proof}
From \(1 + \sqrt{\Delta(R)} \subseteq U(R) = \{1\}\), we deduce that \(\sqrt{\Delta(R)} = \{0\}\). Letting \(0\neq r \in R\) be arbitrary, as \(R\) is \(\sqrt{\Delta}\)-fine, there exist \(a \in \sqrt{\Delta(R)}\) and \(u \in U(R)\) such that \(r = u + a\). But, because of \(\sqrt{\Delta(R)} = \{0\}\), we discover that \(a = 0\), and therefore \(r = u =1\). Thus, \(R = \{0,1\}\), and consequently \(R \cong \mathbb{Z}_2\), as wanted.
\end{proof}

\begin{corollary}\label{co}
If \(R\) is a \(\sqrt{\Delta}\)-fine ring with \(U(R) = 1 + J(R)\), then \(R/J(R) \cong \mathbb{Z}_2\).
\end{corollary}

\begin{proof}
Since \(U(R)/(1+J(R)) \cong U(R/J(R))\), the hypothesis \(U(R) = 1+J(R)\) ensures \(U(R/J(R))\) to be trivial, i.e., \(U(R/J(R)) = \{\bar{1}\}\). However, the quotient \(R/J(R)\) is \(\sqrt{\Delta}\)-fine. Now, applying Proposition~\ref{n}, we extract that \(R/J(R) \cong \mathbb{Z}_2\), as pursued.
\end{proof}

In usage of the standard terminology, the symbol $Nil_*(R)$ denotes the {\it lower nil-radical} of a ring $R$. Recall that a ring $R$ is {\it $2$-primal} if $Nil(R)=Nil_*(R)$. For an arbitrary endomorphism $\alpha$ of $R$, the ring $R$ is said to be $\alpha$-compatible if, for any $a,b \in R$, $ab=0$ if, and only if, $a\alpha(b)=0$. This definition was firstly given in \cite{amin}. In this case, it is readily to see that the map $\alpha$ must be injective.

\medskip

We continue by establishing the following.

\begin{proposition}\label{prim}
Suppose \( R \) is both \( \alpha \)-compatible and \( 2 \)-primal. Then,
\[
\sqrt{\Delta(R[x; \alpha])} = \sqrt{\Delta(R)} + \operatorname{Nil}_*(R)[x; \alpha]x.
\]
\end{proposition}

\begin{proof}
Write \( f = \sum_{i=0}^n a_i x^i \in \sqrt{\Delta(R[x; \alpha])} \). Thus, for some \( m \in \mathbb{N} \), we have \( f^m \in \Delta(R[x; \alpha]) \). Utilizing \cite[Proposition 3.4]{Dj}, it follows that \( a_0 \in \sqrt{\Delta(R)} \) and
\[
a_n \alpha^n(a_n) \cdots \alpha^{n(m-1)}(a_n) \in \operatorname{Nil}_*(R) = \mathrm{Nil}(R),
\]
so \cite[Lemma 2.1]{chen} applies to get that \( a_n \in \mathrm{Nil}(R) = \operatorname{Nil}_*(R) \).

Now, assume \( n \ge 2 \). We menage to show that \( a_{n-1} \in \operatorname{Nil}_*(R) \). To that end, put \( f := g - a_n x^n \). So, one checks that \( f^m = g^m + q \), where \( q \in R[x; \alpha] \). But, each term of \( q \) is a product involving \( a_n \) and some other coefficients. Since \( a_n \in \mathrm{Nil}(R) \) and \( \mathrm{Nil}(R) \) is an ideal, we detect \( q \in \mathrm{Nil}(R)[x; \alpha]x \). Furthermore,
\cite[Proposition 3.4]{Dj} works to get that \( \mathrm{Nil}(R)[x; \alpha]x \subseteq \Delta(R[x; \alpha]) \). Hence, \( g^m \in \Delta(R[x; \alpha]) \). The application of \cite [Proposition 3.4]{Dj} again gives
\[
a_{n-1} \alpha^{n-1}(a_{n-1}) \cdots \alpha^{(n-1)(m-1)}(a_{n-1}) \in \operatorname{Nil}_*(R) = \mathrm{Nil}(R),
\]
so \( a_{n-1} \in \mathrm{Nil}(R) = \operatorname{Nil}_*(R) \). Repeating this argument, we finally obtain \( a_i \in \operatorname{Nil}_*(R) \) for all \( 1 \le i \le n \), as required.

Conversely, assume that $$f=\sum_{i=0}^{n}a_ix^i \in \sqrt{\Delta (R)} + Nil_*(R)[x; \alpha]x.$$ Then, one inspects that $a_0^n \in \Delta(R)$ for some $n \in \mathbb{N}$. Write
\[
f^n = a_0^n + r_1x + \cdots + r_mx^m
\]
for some $m \in \mathbb{N}$. Since, for any $1\le i \le n$, $a_i \in Nil_*(R)$, we can conclude, for each $1\le i \le m$, that $r_i \in Nil_*(R)$.

Assume now $u =\sum_{i=0}^{n}u_ix^i \in U(R[x; \alpha])$. Then, looking at \cite[Lemma 2.1]{che}, we have $u_0 \in U(R)$ and, for each $1\le i \le n$, $u_i \in Nil_*(R)$. Hence, it is plain to see that
\[
1-uf^n \in U(R)+ Nil_*(R)[x; \alpha]x,
\]
so we may conclude from \cite[Lemma 2.1]{che} that $1-uf^n \in U(R[x; \alpha])$. Therefore, $f \in \sqrt{\Delta (R[x; \alpha])}$, as expected.
\end{proof}

\begin{lemma}
Let $R$ be an $\alpha$-compatible ring. Then, the following conditions hold:

(1) $R$ is a $2$-primal ring if, and only if, $\sqrt{\Delta}F(R[x; \alpha])=\sqrt{\Delta}F(R)+Nil_*(R)[x; \alpha]x$.

(2) $R$ is a reduced ring if, and only if, $\sqrt{\Delta}F(R[x; \alpha])=\sqrt{\Delta}F(R)$.
\end{lemma}

\begin{proof}
We only prove (1), because part (2) follows from (1).

To that goal, assume that $R$ is a $2$-primal ring, and let $f=\sum_{i=0}^{n}a_ix^i \in \sqrt{\Delta}F(R[x; \alpha])$. Then, there exist $u=\sum_{i=0}^{n}u_ix^i \in U(R[x; \alpha])$ and $d=\sum_{i=0}^{n}d_ix^i \in \sqrt{\Delta(R[x; \alpha])}$ such that $f=u+d$. From Lemma \ref{prim}, we have $d_0 \in \sqrt{\Delta(R)}$ and, for every $1\le i\le n$, $d_i \in \operatorname{Nil}_*(R)$. Also, from \cite[Lemma 2.1]{che}, we have $u_0 \in U(R)$ and, for every $1\le i\le n$, $u_i \in \operatorname{Nil}_*(R)$. Hence,
\[
f=(u_0+d_0)+\sum_{i=1}^{n}(u_i +d_i)x^i \in \sqrt{\Delta}F(R)+Nil_*(R)[x; \alpha]x.
\]

Conversely, assume that $f=\sum_{i=0}^{n}a_ix^i \in \sqrt{\Delta}F(R)+Nil_*(R)[x; \alpha]x$. Then, there exist $u_0 \in U(R)$ and $d_0 \in \sqrt{\Delta(R)}$ such that $a_0=u_0 + d_0$. Hence, employing Lemma \ref{prim}, we deduce
\[
f=u_0+(d_0+a_1x+\cdots+a_nx^n) \in U(R[x; \alpha])+\sqrt{\Delta(R[x; \alpha])}.
\]
Thus, $f \in \sqrt{\Delta}F(R[x; \alpha])$, as required.

Now reversely, assume that
\[
\sqrt{\Delta}F(R[x; \alpha])=\sqrt{\Delta}F(R)+Nil_*(R)[x; \alpha]x
\]
holds, and let $a \in \operatorname{Nil}(R)$. Then, there exists $n \in \mathbb{N}$ such that $a^n=0$, so from \cite[Lemma 3.1]{chen} we derive
\[
(ax)^n=a\alpha(a)\cdots \alpha^{n-1}(a)x^n=0,
\]
whence $ax \in Nil(R[x; \alpha])$. Therefore, $1+ax \in U(R[x; \alpha]) \subseteq \sqrt{\Delta}F(R[x; \alpha])$. But, by assumption, we infer that $a \in \operatorname{Nil}_*(R)$. Thus, $R$ is a $2$-primal ring after all, as stated.
\end{proof}

\section{Matrix Rings}\label{sec3}

One of the pivotal results in this section records that matrix rings over \(\sqrt{\Delta}\)-fine rings are \(\sqrt{\Delta}\)-fine too. Before coming to it, we need the following preliminaries.

\begin{lemma}\label{2.3}
Let \( R \) be a ring. For every \( n \ge 2 \),
\[
	\Delta(M_n(R)) = J(M_n(R)) = M_n(J(R)).
\]
\end{lemma}

\begin{proof}
Since \( M_n(R) \) is known to be generated by its units for all \( n \ge 2 \), the result follows directly from \cite[Corollary 4]{lm}.
\end{proof}

\begin{proposition}\label{sum}
If \(R = M_n(S)\), where \(n \ge 2\) and \(S\) is any nonzero ring, then every matrix \(M \in R\) is a sum of two $\sqrt{\Delta}$-fine matrices.
\end{proposition}

\begin{proof}
Write \(M = D + T_1 + T_2\), where \(D\) is diagonal, \(T_1\) is strictly upper triangular, and \(T_2\) is strictly lower triangular. Consulting with \cite{Henri}, every diagonal matrix can be expressed as a sum of two invertible diagonal matrices; hence, we can write \(D = U_1 + U_2\) with \(U_1, U_2 \in U(R)\). Now, observe that \(U_1 + T_1\) and \(U_2 + T_2\) are both $\sqrt{\Delta}$-fine matrices. Since $T_1, T_2\in Nil(M_n(R))$ and $Nil(M_n(R))\subseteq \sqrt{\Delta( M_n(R))}$, we extract $T_1, T_2\in \sqrt{\Delta( M_n(R))}$. Consequently,
\[
	M = (U_1 + T_1) + (U_2 + T_2) \in \sqrt{\Delta}F(R) + \sqrt{\Delta}F(R),
\]
as desired.
\end{proof}

\begin{proposition}\label{lemma2}
Let $R$ be a ring, and let $X =(a_{ij})$ be a upper triangular matrix in ${\rm M}_n(R)$. Then, $X \in \sqrt{\Delta({\rm M}_n(R))}$ if, and only if, $a_{ii} \in \sqrt{\Delta(R)}$ for all $1 \le i \le n$.
\end{proposition}

\begin{proof}
Suppose $X \in \sqrt{\Delta({\rm M}_n(R))}$. Then, there is $s \in \mathbb{N}$ such that $$X^s \in \Delta({\rm M}_n(R))= J({\rm M}_n(R))={\rm M}_n(J(R)).$$ This assures that $a^s_{ii} \in J(R)\subseteq \Delta(R)$ for all $1 \le i \le n$, whence $a_{ii} \in \sqrt{\Delta(R)}$ for all $1 \le i \le n$. If, however, $a_{ii} \in \sqrt{\Delta(R)}$ for all $1 \le i \le n$, then, for each $1 \le i \le n$, there is $s_i \in \mathbb{N}$ such that $a^{s_i}_{ii} \in \Delta(R)$. Set $t:=\max\{s_i\mid 1 \le i \le n\}$, and then $a^t_{ii} \in \Delta(R)$ for all $1 \le i \le n$. Now, we calculate
\[
	X^t = \begin{pNiceArray}{cccc}
		a^t_{11} &  & \Block{2-2}<\Large>{\mathbf{\ast}} \\
		& a^t_{22} \\
		\Block{2-2}<\Large>{0} && \ddots &  \\
		&&  & a^t_{nn}
	\end{pNiceArray}
	= \begin{pNiceArray}{cccc}
		a^t_{11} &  & \Block{2-2}<\Large>{0} \\
		& a^t_{22} \\
		\Block{2-2}<\Large>{0} && \ddots &  \\
		&&  & a^t_{nn}
	\end{pNiceArray}
	+ \begin{pNiceArray}{cccc}
		0 &  & \Block{2-2}<\Large>{\mathbf{\ast}} \\
		& 0 \\
		\Block{2-2}<\Large>{0} && \ddots &  \\
		&&  & 0
	\end{pNiceArray}.
\]
Observe that the first right hand side matrix belongs to $\Delta({\rm M}_n(R))$ exploiting \cite[Corollary 9]{lm}, while the second one is manifestly nilpotent. Now, \cite[Lemma 2.1(8)]{DDE} tells us that $X^t \in \sqrt{\Delta({\rm M}_n(R))}$, and hence $X \in \sqrt{\Delta({\rm M}_n(R))}$, as needed.
\end{proof}

\begin{lemma}\label{nil}
Let $R$ be a ring. Then, we have
${\rm Nil}(R) + J(R) \subseteq \sqrt{\Delta(R)}$.
\end{lemma}

\begin{proof}
Choose $q + j \in {\rm Nil}(R) + J(R)$, where $q \in {\rm Nil}(R)$ with $q^n = 0$ and $j \in J(R)$. It is easily seen that $(q + j)^n \in J(R)\subseteq \Delta(R)$, which insures $q + j \in \sqrt{\Delta(R)}$, as expected.
\end{proof}

\begin{proposition}\label{lemma4}
Let $R$ be a ring, and put $A := \begin{pmatrix} X & C \\ 0 & Y \end{pmatrix}$ be a block matrix where $X \in \sqrt{\Delta({\rm M}_s(R))}$ and $Y \in \sqrt{\Delta({\rm M}_t(R))}$. Then, $A \in \sqrt{\Delta({\rm M}_{s+t}(R))}$.
\end{proposition}

\begin{proof}
Since $X \in \sqrt{\Delta({\rm M}_s(R))}$ and $Y \in \sqrt{\Delta({\rm M}_t(R))}$, there are $k_1,k_2 \in \mathbb{N}$ such that $X^{k_1} \in \Delta({\rm M}_s(R))=J({\rm M}_s(R))$ and $Y^{k_2} \in \Delta({\rm M}_t(R))=J({\rm M}_t(R))$. Setting $k := \max\{k_1,k_2\}$, we compute
\[
	A^k = \begin{pmatrix} X^k & 0 \\ 0 & Y^k \end{pmatrix} + \begin{pmatrix} 0 & * \\ 0 & 0 \end{pmatrix} \in J({\rm M}_{s+t}(R)) + \operatorname{Nil}({\rm M}_{s+t}(R)).
\]
Therefore, Lemma \ref{nil} employs to get that $A \in \sqrt{\Delta({\rm M}_{s+t}(R))}$, as intended.
\end{proof}

\begin{proposition}\label{block}
Let
$
	M = \begin{pmatrix} X & C \\ D & Y \end{pmatrix} \in R = M_n(S),
$
where \(X \in M_k(S)\) and \(Y \in M_{n-k}(S)\) with \(1 \le k \le n-1\).
If \(X \in \sqrt{\Delta}F(M_k(S))\) and \(Y \in \sqrt{\Delta}F(M_{n-k}(S))\), then \(M \in \sqrt{\Delta}F(R)\).
\end{proposition}

\begin{proof}
Since \(X\) and \(Y\) are $\sqrt{\Delta}$-fine, we can write
\[
	X = U + A \quad \text{and} \quad Y = U' + A'
\]
for some \(U \in U(M_k(S))\), \(U' \in U(M_{n-k}(S))\), and \(A \in \sqrt{\Delta(M_k(S))}\), \(A' \in \sqrt{\Delta(M_{n-k}(S))}\).
	
Now, consider the decomposition
	
\[
	M = \begin{pmatrix} U & 0 \\ D & U' \end{pmatrix} + \begin{pmatrix} A & C \\ 0 & A' \end{pmatrix}.
	\qquad\qquad\qquad\qquad\qquad (*)
\]
	
One verifies that the first matrix on the right-hand side is block triangular with invertible diagonal blocks, whence it is invertible in \(R\).
The second matrix is in $\sqrt{\Delta(M_n(S))}$ as Proposition \ref{lemma4} suggests. Thus, $(*)$ is a $\sqrt{\Delta}$-fine decomposition of \(M\) in \(R\), proving that \(M \in \sqrt{\Delta}F(R)\), as formulated.
\end{proof}

Further, Proposition \ref{block} applied inductively on the number of diagonal blocks, leads to the following useful fact.

\begin{corollary}\label{Block}
Let \( M = (A_{ij}) \in R = M_n(S) \) be a block matrix with each diagonal block \( A_{ii} \) \(\sqrt{\Delta}\)-fine. Then, \( M \) belongs to \(\sqrt{\Delta}F(R)\).
\end{corollary}

\begin{definition}\label{goodform}\cite[Definition 3.5]{Clam}
For any \( n \ge 2 \), a matrix of the form
$
	\begin{pmatrix} A & \beta \\ \gamma & d \end{pmatrix} \in R = M_n(S)
$
is said to be in {\it good form} whenever \( A \in M_{n-1}(S) \) and \( d \in S \) are both nonzero.
\end{definition}

We, thereby, come to the following principal result.

\begin{theorem}\label{main theorem}
If \( S \) is a \(\sqrt{\Delta}\)-fine ring, then the matrix ring \( M_n(S) \) is also \(\sqrt{\Delta}\)-fine for every \( n \ge 1 \).
\end{theorem}

\begin{proof}
Let \( R = M_n(S) \) and assume that \( S \) is \( \sqrt{\Delta} \)-fine. We argue by induction on \( n \) with the case \( n = 1 \) being trivial. By the induction hypothesis, the ring \( M_{n-1}(S) \) is \( \sqrt{\Delta} \)-fine. 

Next, we first treat the case where \( |S| > 2 \). Regarding Theorem~\ref{2-good}(2), we can write \( 1 \in U(S) + U(S) \). For \( n \ge 2 \), take any nonzero matrix
$$
	M = \begin{pmatrix} A & \beta \\ \gamma & d \end{pmatrix} \in R,
$$
where \( d \in S \) and \( A \in M_{n-1}(S) \). According to \cite[Proposition 3.9]{Clam}, we may assume that \( M \) is in good form; that is, \( A \neq 0 \) and \( d \neq 0 \). By the induction hypothesis, we know \( A \in \sqrt{\Delta}F(M_{n-1}(S)) \). Since \( S \) is \(\sqrt{\Delta}\)-fine, we also have \( d \in \sqrt{\Delta}F(S) \). Furthermore, Proposition~\ref{block} allows us to conclude that \( M \in \sqrt{\Delta}F(R) \). It, thus, remains to consider the case when \( S \cong \mathbb{F}_2 \). But then \cite[Theorem 3.1]{Clam} guarantees that $M_n(\mathbb{F}_2)$ is fine for all $n \ge 1$ and, consequently, it is also $\sqrt{\Delta}$-fine, finishing the argumentation.
\end{proof}

\begin{corollary}\label{artinian}
A ring \(R\) is semi-local \(\sqrt{\Delta}\)-fine if, and only if, it is a simple Artinian ring.
\end{corollary}

\begin{proof}
If, for a moment, \(R\) is simple Artinian, then the Wedderburn–Artin theorem from \cite{lam}, \cite{lamm} yields an isomorphism \(R \cong M_n(S)\) for some division ring \(S\) and some integer \(n \ge 1\). Moreover, division rings are trivially $\sqrt{\Delta}$-fine, and Theorem~\ref{main theorem} teaches us that the matrix ring \(M_n(S)\) inherits $\sqrt{\Delta}$-fineness; hence, \(R\) is $\sqrt{\Delta}$-fine. Besides, every simple Artinian ring is semi-local, so \(R\) is semi-local as well.

Conversely, suppose \(R\) is semi-local and $\sqrt{\Delta}$-fine. Activating Proposition~\ref{simple}, one knows that every nonzero $\sqrt{\Delta}$-fine ring is simple; so, \(R\) is simple and since $J(R)$ is a two-sided ideal of $R$ we thus have that \( J(R) \) is either \( 0 \) or \( R \). But, \( J(R) \neq R \) for any nonzero ring with identity, since \( J(R) \) is always a proper ideal when \( R \) has a nonzero identity. Therefore, $J(R) = 0$. But, since \(R\) is semi-local, the quotient \(R/J(R)\cong R\) is, by definition, a (semi-simple) Artinian ring. Consequently, \(R\) itself is simple Artinian, as suspected.
\end{proof}

Our second major result is the following one.

\begin{theorem}\label{thm}
Let \(G\) be a nonzero abelian group, and put \(R := \operatorname{End}_{\mathbb{Z}}(G)\). Then, the following statements are equivalent:
	
(1) \(R\) is a simple ring;
	
(2) \(R\) is a simple Artinian ring;
	
(3) \(R\) is a \(\sqrt{\Delta}\)-fine ring;
	
(4) \(G\) is a finite-dimensional vector space over either \(\mathbb{Q}\) or \(\mathbb{F}_p\) for some prime \(p\).
\end{theorem}

\begin{proof}
The equivalencies (1) \(\Leftrightarrow\) (2) \(\Leftrightarrow\) (4) are a classical result due to Fuchs (see \cite[Theorem 111.2]{Fuch}). For the implication (3) \(\Rightarrow\) (1), we notice that every \(\sqrt{\Delta}\)-fine ring is simple in regard to Proposition~\ref{simple}.

Conversely, if \(R\) is simple Artinian, then Corollary~\ref{artinian} informs us that \(R\) is \(\sqrt{\Delta}\)-fine, thus yielding (2) \(\Rightarrow\) (3). So, all four conditions are equivalent, as asked for.
\end{proof}

\begin{lemma}\label{semisimple}
A semi-simple ring is \(\sqrt{\Delta}\)-fine if, and only if, it is simple.
\end{lemma}

\begin{proof}
One direction is trivial and was shown in Proposition \ref{simple}; we, therefore, concentrate on the converse one. By the classical Wedderburn-Artin theorem (cf. \cite{lam}, \cite{lamm}), any semi-simple ring \(R\) is isomorphic to a direct product of matrix rings over division rings, i.e.,
$
	R \cong \prod_{i=1}^{k} M_{n_i}(D_i),
$
where each \(D_i\) is a division ring.
If, foremost, \(R\) is simple, then this product must consist of a single factor; whence \(R \cong M_n(D)\) for some division ring \(D\) and some \(n \ge 1\). Viewing Theorem \ref{main theorem}, every matrix ring over a division ring is \(\sqrt{\Delta}\)-fine. Consequently, \(R\) itself is \(\sqrt{\Delta}\)-fine, as requested.
\end{proof}

\begin{lemma}\label{ff}
If \( A = (a_{ij}) \in M_n(S) \) is a \(\sqrt{\Delta}\)-fine matrix, then
$
	\sum_{i,j} S a_{ij} S = S.
$
In particular, \( A = a E_{ij} \in \sqrt{\Delta}F(R) \) yields \( SaS = S \).
\end{lemma}

\begin{proof}
Set \( J := \sum_{i,j} S a_{ij} S \) and \( \overline{S} := S/J \). Then, the natural map \( \pi: M_n(S) \to M_n(\overline{S}) \) sends \( A \) to \( 0 \) since all \( a_{ij} \in J \). As \( A \in \sqrt{\Delta}F(R) \), we write \( A = U + B \) with \( U \in U(M_n(S))\) and \( B\in \sqrt{\Delta(M_n(S))}=\sqrt{J(M_n(S))}\) inspired by Lemma \ref{2.3}. Acting with \( \pi \), we get \( 0 = \pi(U) + \pi(B) \) in \( M_n(\overline{S}) \), where \( \pi(U)\in U(M_n(\overline{S})) \) and \( \pi(B)\in \sqrt{J(M_n(\overline{S}))} \) as Lemma \ref{3} enables us. Therefore, \[ \pi(U)\in U(M_n(\overline{S}))\cap \sqrt{J(M_n(\overline{S}))} \] must be zero adapting
\cite[Proposition 2.3]{Dm}. Hence, \( M_n(\overline{S}) = 0 \), and so \( \overline{S} = 0 \) giving \( J = S \).

As for the special case \( A = a E_{ij} \), we have \( J = SaS \), and the all of above means \( SaS = S \), as we need.
\end{proof}

Our third basic result in this section states the following necessary and sufficient condition.

\begin{theorem}
Let \(R = M_n(S)\), where \(S\) is a commutative ring. If \(aE_{ij} \in \sqrt{\Delta}F(R)\) for some matrix unit \(E_{ij}\), then \(a \in U(S)\). As a consequence, \(R\) is \(\sqrt{\Delta}\)-fine if, and only if, \(S\) is a field.
\end{theorem}

\begin{proof}
Assume \( aE_{ij} \ \in \sqrt{\Delta}F(R) \). Thus, Lemma \ref{ff} works to see that \( aE_{ij} \in \sqrt{\Delta}F(R) \) implies \( SaS = S \). Since \( 1 \in S = S a S \), we can write
\[
1 = x_1 a y_1 + \cdots + x_n a y_n.
\]
Because \( S \) is commutative, this becomes
\[
1 = a(x_1 y_1 + \cdots + x_n y_n) = (x_1 y_1 + \cdots + x_n y_n) a,
\]
which shows that \( a \in U(S) \).
	
''Only if'': Suppose \( R \) is \(\sqrt{\Delta}\)-fine. Take any nonzero \( a \in S \). The matrix \( aE_{11} \) (or any other \( aE_{ij} \)) is nonzero since \( a \neq 0 \) and \( E_{11} \neq 0 \). Therefore, \( aE_{11} \in \sqrt{\Delta}F(R) \). Bearing in mind Part 1, \( a \in U(S) \). Hence, every nonzero element of \( S \) is a unit, which means \( S \) is a field.
	
''If'': This part follows immediately from Theorem \ref{main theorem}, completing the entire arguments.
\end{proof}

\section{Group Rings of $\sqrt{\Delta}$-Fine and Generalized Fine Rings}\label{sec4}

In this section, we pay attention to group rings and investigate when they are generalized fine.

\medskip

Standardly, for any ring \(R\) and any multiplicative group \(G\), we write \(RG\) to stand the group ring of \(G\) over \(R\). The augmentation map \(\varepsilon: RG \to R\) is given by
\[
\varepsilon\left(\sum_{g\in G} a_g g\right) = \sum_{g\in G} a_g.
\]
Its kernel, denoted by \(\varepsilon(RG)\), is called the {\it augmentation ideal} of \(RG\).

On the other hand, a group \(G\) is called a {\it \(p\)-group} if every element has a finite order which is a power of the prime \(p\). Also, \(G\) is said to be {\it locally finite} if every finitely generated subgroup of \(G\) is finite.

\medskip

We start our examination here with the following criterion.

\begin{lemma}\label{8}
Let $R$ be a ring and $G$ a group. Then, $RG$ is a $\sqrt{\Delta}$-fine ring if, and only if, $R$ is a $\sqrt{\Delta}$-fine ring and $G$ is a trivial group.
\end{lemma}

\begin{proof}
It suffices to show that, if $RG$ is a $\sqrt{\Delta}$-fine ring, then $G$ is the trivial group. To that purpose, assume that $1 \neq g \in G$. Then, $1-g = u + d$ for some $u \in U(RG)$ and $d \in \sqrt{\Delta(RG)}$. Applying the augmentation map, we obtain $$-\varepsilon(u) = \varepsilon(d) \in U(R) \cap \sqrt{\Delta(R)} = \varnothing,$$ which contradicts \cite[Remark 2.2(4)]{DDE}. Hence, $G$ must be a trivial group.
\end{proof}

The following statements notably does {\it not} appear in \cite{z}.

\begin{theorem}\label{88}
Let \( R \) be a ring and \( G \) a group. If \( RG \) is a generalized fine ring, then \( R \) is generalized fine and \( G \) is a \( p \)-group with \( p \in J(R) \).
\end{theorem}

\begin{proof}
First, we prove that \( R \) is generalized fine. In fact, take \( a \not\in J(R) \). Since \( J(RG) \cap R \subseteq J(R) \), we have \( a \notin J(RG) \). But, since \( RG \) is generalized fine, we can write \( a = u + t \) with \( u \in U(RG) \) and \( t \in \mathrm{Nil}(RG) \). Applying the augmentation map gives
\[
	a = \varepsilon(a) = \varepsilon(u) + \varepsilon(t) \in U(R) + \mathrm{Nil}(R).
\]
Thus, \( R \) is generalized fine indeed.
	
Next, we show that \( G \) is a \( p \)-group with \( p \in J(R) \). In fact, in virtue of \cite[Proposition 15(i)]{con}, it is enough to establish that \( \varepsilon(RG) \subseteq J(RG) \). Suppose, to the contrary, that there exists \( f \in \varepsilon(RG) \setminus J(RG) \). Since \( RG \) is generalized fine, we can write \( f = u + t \) with \( u \in U(RG) \) and \( t \in \mathrm{Nil}(RG) \). Thus,
\[
	0 = \varepsilon(f) = \varepsilon(u) + \varepsilon(t),
\]
and so \( \varepsilon(u) = -\varepsilon(t) \). But, \( \varepsilon(u) \in U(R) \) and \( \varepsilon(t) \in \mathrm{Nil}(R) \), which is obviously impossible, because a unit cannot be nilpotent. This contradiction means that \( \varepsilon(RG) \subseteq J(RG) \). The result now follows at once from \cite[Proposition 15(i)]{con}.
\end{proof}

Our next results somewhat treats the converse implication which also does {\it not} appear in \cite{z}.

\begin{theorem}\label{888}
Let \( R \) be a ring and \( G \) a group. If \( G \) is a locally finite \( p \)-group with \( p \in \mathrm{Nil}(R) \), and \( R \) is generalized fine, then \( RG \) is also generalized fine.
\end{theorem}

\begin{proof}
Assume \( R \) is generalized fine and \( G \) is a locally finite \( p \)-group such that \( p \in \mathrm{Nil}(R) \). Since \( RG/\varepsilon(RG) \cong R \), the factor-ring \( RG/\varepsilon(RG) \) is generalized fine. Moreover, knowing \cite[Proposition 16]{con}, we deduce that the augmentation ideal \( \varepsilon(RG) \) is a nil-ideal. Thus, \cite[Lemma 2.10]{z} is a guarantor that \( RG \) itself is generalized fine, as we want.
\end{proof}

In closing, we formulate the following challenging question, about which we conjecture that the answer is in the affirmative:

\begin{problem}\label{basic}
Is there a \(\sqrt{\Delta}\)-fine ring which is {\it not} fine?
\end{problem}

Despite of this query, in view of the group rings results presented in this section, especially Lemma~\ref{8} and Theorem~\ref{888}, there is a generalized fine ring that is {\it not} \(\sqrt{\Delta}\)-fine. In fact, to substantiate this claim, we explicitly state the following assertion which improves the example from \cite{z} demonstrating that there is a generalized fine ring that is {\it not} fine.

\begin{example}\label{gfi} There is a generalized fine ring which is {\it not} \(\sqrt{\Delta}\)-fine.
\end{example}

\begin{proof} Suppose $G$ is any non-trivial finite $p$-group and $R$ is a generalized fine ring such that $p$ is nilpotent in $R$. A consultation with Theorem~\ref{888} proposes that the group ring $RG$ is too generalized fine, but Lemma~\ref{8} confirms that $RG$ is surely not \(\sqrt{\Delta}\)-fine as $G\not= \{1\}$, and thus the assertion sustained.
\end{proof}

\section{Concluding Discussion}\label{sec5}

In this work, we introduced the class of $\sqrt{\Delta}$-fine rings as a natural generalization of fine rings. We proved that every such ring is simple (Proposition~\ref{simple}), that abelian ones are indecomposable (Proposition~\ref{centralidempotent}), and, most notably, that $M_n(R)$ remains $\sqrt{\Delta}$-fine for every $n \ge 1$ (Theorem~\ref{main theorem}). This led to a complete characterization of semi-local $\sqrt{\Delta}$-fine rings as exactly the simple Artinian rings (Corollary~\ref{artinian}). We further analyzed group rings, showing that $RG$ is $\sqrt{\Delta}$-fine only when $G$ is trivial (Lemma~\ref{8}), and obtained conditions for generalized fine group rings (Theorems~\ref{88} and~\ref{888}). However, whether every $\sqrt{\Delta}$-fine ring is necessarily fine remains an open problem (Problem~\ref{basic}), which we conjecture to have an affirmative answer. Nevertheless, we present an example of a generalized fine ring that is {\it not} $\sqrt{\Delta}$-fine (Example~\ref{gfi}).

\medskip
\medskip

\noindent{\bf Acknowledgement.} The authors express their sincere gratitude to the expert referee for the numerous competent suggestions made which led to a substantial improvement of the exposition.

\medskip
\medskip

\noindent{\bf Funding:} This work is based upon research funded by Iran National Science Foundation (INSF) under Project no. 40502582.

\end{document}